\documentclass{amsart}

\usepackage{amsmath}
\usepackage{amssymb}
\usepackage{xcolor}

\allowdisplaybreaks

\newtheorem{thm}{Theorem}[section]
\newtheorem{cor}[thm]{Corollary}
\newtheorem{lem}[thm]{Lemma}
\newtheorem{propo}[thm]{Proposition}

\theoremstyle{definition}

\theoremstyle{remark}

\title[A transplantation theorem]{A weighted transplantation theorem for Jacobi coefficients}

\author[A. Arenas]{Alberto Arenas}
\address{Departamento de Matem\'aticas y Computaci\'on,
Universidad de La Rioja, Complejo Cient\'{\i}fico-Tecnol\'ogico,
Calle Madre de Dios 53, 26006 Logro\~no, Spain}
\email{alberto.arenas@unirioja.es}

\author[\'O. Ciaurri]{\'Oscar Ciaurri}
\address{Departamento de Matem\'aticas y Computaci\'on,
Universidad de La Rioja, Complejo Cient\'{\i}fico-Tecnol\'ogico,
Calle MAdre de Dios 53, 26006 Logro\~no, Spain}
\email{oscar.ciaurri@unirioja.es}

\author[E. Labarga]{Edgar Labarga}
\address{Departamento de Matem\'aticas y Computaci\'on,
Universidad de La Rioja, Complejo Cient\'{\i}fico-Tecnol\'ogico,
Calle Madre de Dios 53, 26006 Logro\~no, Spain}
\email{edgar.labarga@unirioja.es}

\keywords{Transplantation, Jacobi coefficients, Jacobi polynomials, discrete vector-valued local Calder\'{o}n-Zygmund theory, weighted norm inequalities}
\subjclass[2010]{Primary: 42C10.  Secondary: 44A20}
\thanks{The first-named author was supported by a predoctoral research grant of the Government of Comunidad Aut\'{o}noma de La Rioja. The second-named author was supported by grant MTM2015-65888-C04-4-P MINECO/FEDER, UE, from Spanish Government. The third-named author was supported by a predoctoral research grant of the University of La Rioja.}

\begin{document}

\begin{abstract}
We present a transplantation theorem for Jacobi coefficients in weighted spaces. In fact, by using a discrete vector-valued local Calder\'{o}n-Zygmund theory, which has recently been furnished, we prove the boundedness of transplantation operators from $\ell^p(\mathbb{N},w)$ into itself, where $w$ is a weight in the discrete Muckenhoupt class $A_{p}(\mathbb{N})$. Moreover, we obtain weighted weak $(1,1)$ estimates for those operators.
\end{abstract}

\maketitle

\section{Introduction}
The Jacobi polynomials $P^{(\alpha,\beta)}_n(z)$ are defined by means of the Rodrigues' formula (see \cite[p.~67, eq.~(4.3.1)]{Szego}) by
\[
(1-z)^{\alpha}(1+z)^{\beta}P_n^{(\alpha,\beta)}(z)=\frac{(-1)^n}{2^n \, n!}\frac{d^n}{dz^n}\left((1-z)^{\alpha+n}(1+z)^{\beta+n}\right),
\]
where $n$ is in $\mathbb{N}=\{0,1,2,\ldots\}$ and $\alpha,\beta>-1$. It is a well-known fact that they are orthogonal on the interval $(-1,1)$ with respect to the measure  $(1-z)^\alpha(1+z)^{\beta}\, dz$.

Let us consider the family of functions\footnote{These functions are sometimes referred to as Jacobi functions. However, this term is confusing and we will drop it.} $\{p_n^{(\alpha,\beta)}(x)\}_{n\ge 0}$ given by
\[
p_n^{(\alpha,\beta)}(x)=w_n^{(\alpha,\beta)}\Big(\sin\frac{x}{2}\Big)^{\alpha+1/2}\Big(\cos\frac{x}{2}\Big)^{\beta+1/2}P_n^{(\alpha,\beta)}(\cos x), \qquad 0< x < \pi,
\]
where $w_n^{(\alpha,\beta)}$ is the normalization factor defined by
\begin{equation*}
w_n^{(\alpha,\beta)}= \sqrt{\frac{(2n+\alpha+\beta+1)\, \Gamma(n+1)\,\Gamma(n+\alpha+\beta+1)}{\Gamma(n+\alpha+1)\,\Gamma(n+\beta+1)}},\qquad n\geq1,
\end{equation*}
and
\[
w_{0}^{(\alpha,\beta)} = \sqrt{\frac{\Gamma(\alpha+\beta+2)}{\Gamma(\alpha+1)\Gamma(\beta+1)}}.
\]
This family is a complete orthonormal system in the space $L^2(0,\pi)$, the set of all measurable and square integrable functions on $(0,\pi)$ with respect to the Lebesgue measure.

In this situation, for each sequence $\{f(n)\}_{n\ge 0}$ in $\ell^2(\mathbb{N})$, we define the $(\alpha,\beta)$-transform $\mathcal{F}_{\alpha,\beta}$ by the identity
\[
\mathcal{F}_{\alpha,\beta}f(x)=\sum_{m=0}^{\infty}f(m)p_m^{(\alpha,\beta)}(x).
\]
It turns out that $\mathcal{F}_{\alpha,\beta}f$ is a function in $L^2(0,\pi)$. In fact, $\mathcal{F}_{\alpha,\beta}$ is an isometry from $\ell^2(\mathbb{N})$ into $L^2(0,\pi)$ and its inverse is given by
\[
\mathcal{F}_{\alpha,\beta}^{-1}h(n)=\int_{0}^{\pi}h(x)p_n^{(\alpha,\beta)}(x)\, dx, \qquad h\in L^2(0,\pi).
\]

Then, taking $f\in \ell^2(\mathbb{N})$, we define the transplantation operator by
\[
T_{\alpha,\beta}^{\gamma,\delta}f(n)=\mathcal{F}_{\gamma,\delta}^{-1}(\mathcal{F}_{\alpha,\beta}f)(n).
\]
Note that $T_{\alpha,\beta}^{\gamma,\delta}$ is an isometric isomorphism on $\ell^{2}(\mathbb{N})$ and becomes the identity operator when $\alpha=\gamma$ and $\beta=\delta$.

The study of transplantation theorems in Fourier analysis has a long and rich history. The first transplantation result \cite{Guy} is due to D. L. Guy and dates back to the early sixties. Several years later, R. Askey and S. Wainger published a collection of works developing the transplantation theory for Jacobi expansions (see \cite{CAskeyWainger,DAskeyWainger,DAskey,CAskey}). The importance of them is that the research done was twofold: the first part deals with the continuous context (see \cite{CAskeyWainger,CAskey}) and the remaining focuses on the discrete counterpart (\cite{DAskeyWainger,DAskey}). The results in the continuous context were extended by B. Muckenhoupt, including weights, in \cite{Muckenhoupt} (see also \cite{CNS} for a different approach). The present document is related to the discrete setting.

Guy's result initiated an intensive research line which is still in progress and in which many important mathematicians were and are involved. The interested reader is urged to consult the excellent survey \cite{Stempaksurvey} (and the references therein), although it is focused mainly on the continuous setting.

It turns out that the discrete setting has considerably less fruitful results. Apart from Askey and Wainger work on ultraspherical coefficients \cite{DAskeyWainger}, Askey on Jacobi coefficients \cite{DAskey}, and the work of K. Stempak \cite{Stempak} on Fourier-Bessel coefficients, there seems to be no noteworthy track in the literature. Recently, J.J. Betancor et al. \cite{Bet-et-al} have generalised the results of Askey and Wainger for ultraspherical coefficients considering general weights. However, the results are not stated in a whole natural range of parameters as in \cite{DAskeyWainger} and some additional restrictions are imposed.

It has already been mentioned that our present work fits into transplantation theory from a discrete point of view. Particularly, our setting is that of Jacobi coefficients and our main result is a transplantation theorem for $\mathcal{T}_{\alpha,\beta}^{\gamma,\delta}$ in weighted $\ell^{p}$-spaces with fairly general weights. In this way, we generalise the transplantation results in \cite{DAskey} and \cite{Bet-et-al} by allowing general weights and a whole natural range of parameters (see the details in Theorem~\ref{thm:trans} below). Our work is highly motivated by those papers, and some ideas appearing in our proofs are inspired by them. In particular, we use the theory developed in \cite{Bet-et-al} concerning discrete vector-valued Calder\'{o}n-Zygmund operators as a crucial tool.

Before formulating our results, we need some previous definitions. A weight on $\mathbb{N}$ will be a strictly positive sequence $w=\{w(n)\}_{n\ge 0}$. We consider the weighted $\ell^{p}$-spaces
\[
\ell^p(\mathbb{N},w)=\left\{f=\{f(n)\}_{n\ge 0}: \|f\|_{\ell^{p}(\mathbb{N},w)}:=\Bigg(\sum_{m=0}^{\infty}|f(m)|^p w(m)\Bigg)^{1/p}<\infty\right\},
\]
$1\le p<\infty$, and the weak weighted $\ell^{1}$-space
\[
\ell^{1,\infty}(\mathbb{N},w)=\left\{f=\{f(n)\}_{n\ge 0}: \|f\|_{\ell^{1,\infty}(\mathbb{N},w)}:=\sup_{t>0}t\sum_{\{m\in \mathbb{N}: |f(m)|>t\}} w(m)<\infty\right\}.
\]
We simply write $\ell^p(\mathbb{N})$ and $\ell^{1,\infty}(\mathbb{N})$ when $w(n)=1$ for all $n\in \mathbb{N}$.

Furthermore, we say that a weight $w(n)$ belongs to the discrete Muckenhoupt $A_p(\mathbb{N})$ class, $1<p<\infty$, provided that
\[
[w]_{A_p(\mathbb{N})}:=\sup_{\begin{smallmatrix} 0\le n \le m \\ n,m\in \mathbb{N} \end{smallmatrix}} \frac{1}{(m-n+1)^p}\Bigg(\sum_{k=n}^mw(k)\Bigg)\Bigg(\sum_{k=n}^mw(k)^{-1/(p-1)}\Bigg)^{p-1} <\infty,
\]
and that $w(n)$ belongs to the discrete Muckenhoupt $A_1(\mathbb{N})$ class if
\[
[w]_{A_1(\mathbb{N})}:=\sup_{\begin{smallmatrix} 0\le n \le m \\ n,m\in \mathbb{N} \end{smallmatrix}} \frac{1}{m-n+1}\Bigg(\sum_{k=n}^mw(k)\Bigg)\max_{n\le k \le m}w(k)^{-1} <\infty,
\]
holds.

The main result of the paper reads as follow.
\begin{thm}
\label{thm:trans}
Let $\alpha,\beta,\gamma,\delta\ge -1/2$, with $\alpha\neq \gamma$ and $\beta\neq\delta$.
\begin{enumerate}
\item[i)]
If $1<p<\infty$ and $w\in A_p(\mathbb{N})$, then
\begin{equation}
\label{eq:bound-trans}
\|T_{\alpha,\beta}^{\gamma,\delta}f\|_{\ell^p(\mathbb{N},w)}\le C \|f\|_{\ell^p(\mathbb{N},w)}, \qquad f\in\ell^2(\mathbb{N})\cap\ell^{p}(\mathbb{N},w),
\end{equation}
where $C$ is a constant independent of $f$. Consequently, the operator $\mathcal{T}_{\alpha,\beta}^{\gamma,\delta}$ extends uniquely to a bounded linear operator from $\ell^p(\mathbb{N},w)$ into itself.
\item[ii)]
If $w\in A_1(\mathbb{N})$, then
\begin{equation}
\label{eq:bound-trans-weak}
\|T_{\alpha,\beta}^{\gamma,\delta}f\|_{\ell^{1,\infty}(\mathbb{N},w)}\le C \|f\|_{\ell^1(\mathbb{N},w)}, \qquad f\in \ell^2(\mathbb{N})\cap\ell^{1}(\mathbb{N},w),
\end{equation}
where $C$ is a constant independent of $f$. Consequently, the operator $\mathcal{T}_{\alpha,\beta}^{\gamma,\delta}$ extends uniquely to a bounded linear operator from $\ell^1(\mathbb{N},w)$ into $\ell^{1,\infty}(\mathbb{N},w)$.
\end{enumerate}
\end{thm}

An immediate consequence of this theorem is the following result.
\begin{cor}
Let $1<p<\infty$, $w\in A_{p}(\mathbb{N})$, and $\alpha,\beta,\gamma,\delta\ge -1/2$, with $\alpha\neq \gamma$ and $\beta\neq\delta$. There exist a constant $C$ such that
\begin{equation*}
\frac{1}{C} \| f\|_{\ell^{p}(\mathbb{N},w)} \leq \|T_{\alpha,\beta}^{\gamma,\delta} f\|_{\ell^{p}(\mathbb{N},w)} \leq C \| f\|_{\ell^{p}(\mathbb{N},w)},\qquad f\in \ell^{p}(\mathbb{N},w).
\end{equation*}
\end{cor}

Our main result is, in fact, a transplantation theorem. It extends \cite[Theorem 1]{CAskey} including general weights (in that paper only the weights of the form $(n+1)^\sigma$ were considered) and functions $F\in L^q_{v}(0,\pi)$, where $L^q_v(0,\pi)$  is the set of measurable functions such that the norm
\[
\|F\|_{L^q_v(0,\pi)}=\left(\int_{0}^{\pi}|F(x)|^qv(x)\, dx\right)^{1/q},\qquad 1<q<\infty,
\]
is finite. It is known (see \cite{Ker,ABetal}) that
\begin{equation}
\label{eq:conv}
S^{(\alpha,\beta)}_MF=\sum_{m=0}^M a_m^{(\alpha,\beta)}(F)p_m^{(\alpha,\beta)}\longrightarrow F \qquad \text{ in $L^q_v(0,\pi)$}
\end{equation}
for each $F\in L^q_v(0,\pi)$, where
\[
a_m^{(\alpha,\beta)}(F)=\int_{0}^{\pi} F(x)p_m^{(\alpha,\beta)}(x)\, dx
\]
and $v\in A_q(0,\pi)$. By means of a duality argument, from \eqref{eq:conv}, we obtain that
\[
\int_{0}^{\pi}v^{1-q'}(x)\,dx<\infty,
\]
where $1/q+1/q'=1$. Then, from this fact and using the bound $|p_m^{(\gamma,\delta)}(x)|\le C$ for $\gamma,\delta\ge -1/2$ (see \eqref{eq:unif-bound} below), we deduce that
\[
a_n^{(\gamma,\delta)}(F)=\lim_{M\to \infty}\int_{0}^{\pi} S_M^{(\alpha,\beta)}F(x)p_n^{(\gamma,\delta)}(x)\, dx.
\]
In this way, by using \eqref{eq:trans} below,
\[
a_n^{(\gamma,\delta)}(F)=\lim_{M\to \infty}\sum_{m=0}^{M} a_m^{(\alpha,\beta)}(F)\int_{0}^{\pi}p_n^{(\gamma,\delta)}(x)p_m^{(\alpha,\beta)}(x)\, dx=
T_{\alpha,\beta}^{\gamma,\delta}f(n),
\]
where $f(n)=a_n^{(\alpha,\beta)}(F)$. The previous argument proves the following corollary.

\begin{cor}
Let $\alpha,\beta,\gamma,\delta\ge -1/2$, with $\alpha\neq \gamma$ and $\beta\neq\delta$, $1<p,q<\infty$, $v\in A_q(0,\pi)$, and $w\in A_p(\mathbb{N})$. There exist a constant $C$ such that
\begin{equation*}
\frac{1}{C} \| a_n^{(\alpha,\beta)}(F)\|_{\ell^{p}(\mathbb{N},w)} \leq \| a_n^{(\gamma,\delta)}(F)\|_{\ell^{p}(\mathbb{N},w)} \leq C \| a_n^{(\alpha,\beta)}(F)\|_{\ell^{p}(\mathbb{N},w)},\qquad F\in L_v^q(0,\pi).
\end{equation*}
\end{cor}

\section{Proof of Theorem \ref{thm:trans}}
The main tool to prove Theorem \ref{thm:trans} is a discrete local Calder\'on-Zygmung theory developed in \cite{Bet-et-al}. For the reader's convenience, it is appropriate to recall some of the basic aspects of this theory.

Suppose that $\mathbb{B}_1$ and $\mathbb{B}_2$ are Banach spaces. We denote by $\mathcal{L}(\mathbb{B}_1,\mathbb{B}_2)$ the space of bounded linear operators from $\mathbb{B}_1$ into $\mathbb{B}_2$. Let us suppose that
\[
K:(\mathbb{N}\times\mathbb{N})\setminus D \longrightarrow \mathcal{L}(\mathbb{B}_1,\mathbb{B}_2),
\]
where $D:=\{(n,n):n\in \mathbb{N}\}$, is measurable and that for certain positive constant $C$ and for each $n$, $m$, $l\in \mathbb{N}$, $n\neq m$, the following conditions hold.
\begin{enumerate}
\item[(a)] the size condition:
\[
\|K(n,m)\|_{\mathcal{L}(\mathbb{B}_1,\mathbb{B}_2)}\le \frac{C}{|n-m|},
\]
\item[(b)] the regularity properties:
\begin{enumerate}
\item[(b1)]
\[
\|K(n,m)-K(l,m)\|_{\mathcal{L}(\mathbb{B}_1,\mathbb{B}_2)}\le C \frac{|n-l|}{|n-m|^2},\quad |n-m|>2|n-l|, \frac{m}{2}\le n,l\le 2m,
\]
\item[(b2)]
\[
\|K(m,n)-K(m,l)\|_{\mathcal{L}(\mathbb{B}_1,\mathbb{B}_2)}\le C \frac{|n-l|}{|n-m|^2},\quad |n-m|>2|n-l|, \frac{m}{2}\le n,l\le 2m.
\]
\end{enumerate}
\end{enumerate}
A kernel $K$ satisfying conditions (a) and (b) is called a local $\mathcal{L}(\mathbb{B}_1,\mathbb{B}_2)$-standard kernel. For a Banach space $\mathbb{B}$ and a weight $w=\{w(n)\}_{n\ge 0}$, we consider the space
\[
\ell^{p}_{\mathbb{B}}(\mathbb{N},w)=\left\{ \text{$\mathbb{B}$-valued sequences } f=\{f(n)\}_{n\ge 0}: \{\|f(n)\|_{\mathbb{B}}\}_{n\ge 0}\in \ell^p(\mathbb{N},w)\right\}
\]
for $1\le p<\infty$, and
\[
\ell^{1,\infty}_{\mathbb{B}}(\mathbb{N},w)=\left\{ \text{$\mathbb{B}$-valued sequences } f=\{f(n)\}_{n\ge 0}: \{\|f(n)\|_{\mathbb{B}}\}_{n\ge 0}\in \ell^{1,\infty}(\mathbb{N},w)\right\}.
\]
As usual, we simply write $\ell_{\mathbb{B}}^r(\mathbb{N})$ and $\ell^{1,\infty}_{\mathbb{B}}(\mathbb{N})$ when $w(n)=1$ for all $n\in \mathbb{N}$. Also, by $\mathbb{B}_0^{\mathbb{N}}$ we represent the space of $\mathbb{B}$-valued sequences $f=\{f(n)\}_{n\ge 0}$ such that $f(n)=0$, with $n>j$, for some $j\in \mathbb{N}$.

The next theorem is an extension result for Calder\'{o}n-Zygmund operators with standard kernels.
\begin{thm}[Theorem 2.1 in \cite{Bet-et-al}]
\label{thm:CZ}
Let $\mathbb{B}_1$ and $\mathbb{B}_2$ be Banach spaces. Suppose that $T$
is a linear and bounded operator from $\ell_{\mathbb{B}_1}^r(\mathbb{N})$ into $\ell_{\mathbb{B}_2}^r(\mathbb{N})$, for some $1<r<\infty$, and such that there exists a local $\mathcal{L}(\mathbb{B}_1,\mathbb{B}_2)$-standard kernel $K$ such that, for every sequence $f\in (\mathbb{B}_1)_0^{\mathbb{N}}$,
\[
Tf(n)=\sum_{m=0}^{\infty}K(n,m)\cdot f(m),
\]
for every $n\in \mathbb{N}$ such that $f(n)=0$. Then,
\begin{enumerate}
\item[(i)] for every $1< p <\infty$ and $w\in A_p(\mathbb{N})$ the operator $T$
can be extended from $\ell_{\mathbb{B}_1}^r(\mathbb{N})\cap \ell_{\mathbb{B}_1}^p(\mathbb{N},w)$
to $\ell_{\mathbb{B}_1}^p(\mathbb{N},w)$ as a bounded operator from $\ell_{\mathbb{B}_1}^p(\mathbb{N},w)$ into $\ell_{\mathbb{B}_2}^p(\mathbb{N},w)$.

\item[(ii)] for every $w\in A_1(\mathbb{N})$ the operator $T$
can be extended from $\ell_{\mathbb{B}_1}^r(\mathbb{N})\cap \ell_{\mathbb{B}_1}^1(\mathbb{N},w)$
to $\ell_{\mathbb{B}_1}^1(\mathbb{N},w)$ as a bounded operator from $\ell_{\mathbb{B}_1}^p(\mathbb{N},w)$ into $\ell_{\mathbb{B}_2}^{1,\infty}(\mathbb{N},w)$.
\end{enumerate}
\end{thm}

To complete the proof of our result we will need the following fact related to weights in $A_p(\mathbb{N})$.

\begin{lem}
\label{lem:weight}
  Let $1\le p <\infty$ and $w\in A_p(\mathbb{N})$. Then, $w(n)\simeq w(n+1)$.
\end{lem}
\begin{proof}
For $1<p<\infty$ and $w\in A_p(\mathbb{N})$, it is clear that
\[
[w]_{A_p(\mathbb{N})}\ge \frac{1}{2^p}(w(n)+w(n+1))(w(n)^{-1/(p-1)}+w(n+1)^{-1/(p-1)})^{p-1}, \qquad n\in \mathbb{N}.
\]
Now, by means of the inequality $(a+b)^r\ge C_r (a^r+b^r)$, where $a,b,r>0$ and $C_r=\min\{2^{r-1},1\}$, we have
\begin{equation*}
[w]_{A_p(\mathbb{N})}\ge \frac{C_{p-1}}{2^p}(w(n)+w(n+1))(w(n)^{-1}+w(n+1)^{-1})> \frac{C_{p-1}}{2^p}w(n)w(n+1)^{-1}
\end{equation*}
and, similarly,
\[
[w]_{A_p(\mathbb{N})}> \frac{C_{p-1}}{2^p}w(n+1)w(n)^{-1}.
\]
So,
\[
\frac{C_{p-1}}{2^p [w]_{A_p(\mathbb{N})}}w(n)< w(n+1)< \frac{2^p [w]_{A_p(\mathbb{N})}}{C_{p-1}}w(n).
\]

For $p=1$, if we suppose first $w(n)\le w(n+1)$, then it is clear that
\begin{align*}
[w]_{A_1(\mathbb{N})}&\ge \frac{1}{2}(w(n)+w(n+1))\max\{(w(n)^{-1},w(n+1)^{-1}\}\\&=\frac{1}{2}\left(1+w(n+1)w(n)^{-1}\right)> \frac{w(n+1)w(n)^{-1}}{2}
\end{align*}
and we obtain $w(n)\le w(n+1)<2[w]_{A_1(\mathbb{N})} w(n)$. By other hand, supposing $w(n+1)< w(n)$ the procedure is exactly the same.
\end{proof}

We give now a proof of Theorem~\ref{thm:trans}. First, supposing $f\in \mathbb{C}_{0}^{\mathbb{N}}$, we can express the transplantation operator as
\begin{equation}
\label{eq:trans}
T_{\alpha,\beta}^{\gamma,\delta}f(n)=\sum_{m=0}^{\infty}f(m){K}_{\alpha,\beta}^{\gamma,\delta}(n,m),
\end{equation}
with
\[
K_{\alpha,\beta}^{\gamma,\delta}(n,m)=\int_{0}^{\pi}p_n^{(\gamma,\delta)}(x)p_m^{(\alpha,\beta)}(x)\, dx.
\]

Now, we note that it is possible to split the $m$ variable in its even and odd parts, so we have
\begin{equation*}
T_{\alpha,\beta}^{\gamma,\delta} f(n) = \sum_{m=0}^{\infty} f(2m) K_{\alpha,\beta}^{\gamma,\delta}(n,2m) + \sum_{m=0}^{\infty} f(2m+1) K_{\alpha,\beta}^{\gamma,\delta}(n,2m+1),
\end{equation*}
which motivates the following defintions
\begin{align*}
{}^{\text{e,e}}T_{\alpha,\beta}^{\gamma,\delta}f(n) &= \sum_{m=0}^{\infty} f(m) {}^{\text{e,e}}K_{\alpha,\beta}^{\gamma,\delta}(n,m), & {}^{\text{e,e}}K_{\alpha,\beta}^{\gamma,\delta}(n,m) &= K_{\alpha,\beta}^{\gamma,\delta}(2n,2m),\\
{}^{\text{e,o}}T_{\alpha,\beta}^{\gamma,\delta}f(n) &= \sum_{m=0}^{\infty} f(m) {}^{\text{e,o}}K_{\alpha,\beta}^{\gamma,\delta}(n,m), & {}^{\text{e,o}}K_{\alpha,\beta}^{\gamma,\delta}(n,m) &= K_{\alpha,\beta}^{\gamma,\delta}(2n,2m+1),\\
{}^{\text{o,e}}T_{\alpha,\beta}^{\gamma,\delta}f(n) &= \sum_{m=0}^{\infty} f(m) {}^{\text{o,e}}K_{\alpha,\beta}^{\gamma,\delta}(n,m), & {}^{\text{o,e}}K_{\alpha,\beta}^{\gamma,\delta}(n,m) &= K_{\alpha,\beta}^{\gamma,\delta}(2n+1,2m),
\intertext{and}
{}^{\text{o,o}}T_{\alpha,\beta}^{\gamma,\delta}f(n) &= \sum_{m=0}^{\infty} f(m) {}^{\text{o,o}}K_{\alpha,\beta}^{\gamma,\delta}(n,m), & {}^{\text{o,o}}K_{\alpha,\beta}^{\gamma,\delta}(n,m) &= K_{\alpha,\beta}^{\gamma,\delta}(2n+1,2m+1).
\end{align*}
In this way, we obtain that
\begin{equation*}
T_{\alpha,\beta}^{\gamma,\delta} f(2n) = {}^{\text{e,e}}T_{\alpha,\beta}^{\gamma,\delta} \tilde{f}(n) + {}^{\text{e,o}}T_{\alpha,\beta}^{\gamma,\delta} \hat{f}(n),
\end{equation*}
and
\begin{equation*}
T_{\alpha,\beta}^{\gamma,\delta} f(2n+1) = {}^{\text{o,e}}T_{\alpha,\beta}^{\gamma,\delta} \tilde{f}(n) + {}^{\text{o,o}}T_{\alpha,\beta}^{\gamma,\delta} \hat{f}(n),
\end{equation*}
with $\tilde{f}(n) = f(2n)$ and $\hat{f}(n) = f(2n+1)$, $n\in\mathbb{N}$. In addition, note that $^{\text{e,e}}T_{\alpha,\beta}^{\gamma,\delta}$, $^{\text{e,o}}T_{\alpha,\beta}^{\gamma,\delta}$, $^{\text{o,e}}T_{\alpha,\beta}^{\gamma,\delta}$, and $^{\text{o,o}}T_{\alpha,\beta}^{\gamma,\delta}$ are bounded operators in $\ell^{2}(\mathbb{N})$ because so is $T_{\alpha,\beta}^{\gamma,\delta}$. Indeed, let us define the functions
\begin{equation*}
g(n) = f(n/2)\chi_{\mathcal{E}}(n)
\qquad \text{ and } \qquad
h(n) = f((n-1)/2)\chi_{\mathcal{O}}(n),
\end{equation*}
where $\mathcal{E}$ and $\mathcal{O}$ denotes the sets of even and odd numbers respectively. Then, it is verified that  ${}^{\text{e,e}}T_{\alpha,\beta}^{\gamma,\delta} f(n) = T_{\alpha,\beta}^{\gamma,\delta} g(2n)$, ${}^{\text{e,o}}T_{\alpha,\beta}^{\gamma,\delta} f(n) = T_{\alpha,\beta}^{\gamma,\delta} h(2n)$, ${}^{\text{o,e}}T_{\alpha,\beta}^{\gamma,\delta} f(n) = T_{\alpha,\beta}^{\gamma,\delta} g(2n+1)$, and ${}^{\text{o,o}}T_{\alpha,\beta}^{\gamma,\delta} f(n) = T_{\alpha,\beta}^{\gamma,\delta} h(2n+1)$, so the boundedness in $\ell^2(\mathbb{N})$ of each operator follows immediately.

Therefore, it is enough to prove that the kernels ${}^{\text{e,e}}K_{\alpha,\beta}^{\gamma,\delta}$, ${}^{\text{e,o}}K_{\alpha,\beta}^{\gamma,\delta}$, ${}^{\text{o,e}}K_{\alpha,\beta}^{\gamma,\delta}$, and ${}^{\text{o,o}}K_{\alpha,\beta}^{\gamma,\delta}$ satisfy the properties (a) and (b). These facts are consequence of the following propositions (see~\cite{ACL-JacI} for similar estimates in other setting).
\begin{propo}
\label{propo:size}
  Let $n,m\in\mathbb{N}$, $n\neq m$, $\alpha,\beta,\gamma,\delta \ge -1/2$, and $\alpha\neq\gamma$, $\beta\neq\delta$. Then,
\begin{equation}
\label{eq:size-cond}
|K_{\alpha,\beta}^{\gamma,\delta}(n,m)|\le \frac{C}{|n-m|}.
\end{equation}
\end{propo}

\begin{propo}
\label{propo:reg}
  Let $n,m\in\mathbb{N}$, $n\neq m$, $m/2\leq n\leq 2m$, $\alpha,\beta,\gamma,\delta \ge -1/2$, and $\alpha\neq\gamma$, $\beta\neq\delta$. Then,
\begin{equation}
\label{eq:reg-cond-1}
|K_{\alpha,\beta}^{\gamma,\delta}(n+2,m)-K_{\alpha,\beta}^{\gamma,\delta}(n,m)|\le \frac{C}{|n-m|^{2}},
\end{equation}
and
\begin{equation}
\label{eq:reg-cond-2}
|K_{\alpha,\beta}^{\gamma,\delta}(n,m)-K_{\alpha,\beta}^{\gamma,\delta}(n,m+2)|\le \frac{C}{|n-m|^{2}}.
\end{equation}
\end{propo}

In this way, by Theorem \ref{thm:CZ} and taking the weights $w_o(n)=w(2n+1)$ and $w_e(n)=w(2n)$ (note that both of them belongs to $A_p(\mathbb{N})$ because $w\in A_p(\mathbb{N})$), for $1<p<\infty$ we have
\[
\|{}^{\text{e,e}}T_{\alpha,\beta}^{\gamma,\delta} \tilde{f}\|_{\ell^p(\mathbb{N},w_e)}\le \|\tilde{f}\|_{\ell^p(\mathbb{N},w_e)},
\]
\[
\|{}^{\text{e,o}}T_{\alpha,\beta}^{\gamma,\delta} \hat{f}\|_{\ell^p(\mathbb{N},w_e)}\le \|\hat{f}\|_{\ell^p(\mathbb{N},w_e)},
\]
\[
\|{}^{\text{o,e}}T_{\alpha,\beta}^{\gamma,\delta} \tilde{f}\|_{\ell^p(\mathbb{N},w_o)}\le \|\tilde{f}\|_{\ell^p(\mathbb{N},w_o)},
\]
\[
\|{}^{\text{o,o}}T_{\alpha,\beta}^{\gamma,\delta} \hat{f}\|_{\ell^p(\mathbb{N},w_o)}\le \|\hat{f}\|_{\ell^p(\mathbb{N},w_o)},
\]
and the corresponding weak inequalities for $p=1$. To complete the proof, it is enough to observe that, by Lemma \ref{lem:weight},
\[
\|\hat{f}\|_{\ell^p(\mathbb{N},w_e)}\le C \|\hat{f}\|_{\ell^p(\mathbb{N},w_o)}\le C\|f\|_{\ell^p(\mathbb{N},w)}
\]
and
\[
\|\tilde{f}\|_{\ell^p(\mathbb{N},w_o)}\le C \|\hat{f}\|_{\ell^p(\mathbb{N},w_e)}\le C\|f\|_{\ell^p(\mathbb{N},w)}.
\]
\section{Proof of Proposition \ref{propo:size}}
The functions $p_n^{(a,b)}$, $n\in\mathbb{N}$ and $a,b>-1$, are eigenfunctions of the second order differential operator (see \cite[p.~67, eq.~(4.24.2)]{Szego})
\[
L^{a,b}=-\frac{d^2}{dx^2}-\left(\frac{1-4a^2}{16\sin^2(x/2)}+\frac{1-4b^2}{16\cos^2(x/2)}\right),\qquad x\in(0,\pi),
\]
that is,
\begin{equation}
\label{eq:eigen-eq}
L^{a,b}p_n^{(a,b)}=\lambda_n^{(a,b)}p_n^{(a,b)}
\end{equation}
with the eigenvalues given by $\lambda_n^{(a,b)}=(n+(a+b+1)/2)^2$. The operator $L^{a,b}$ is self-adjoint in $L^2(0,\pi)$ and furthermore, for any interval $[r,s]\subset (0,\pi)$, the identity
\begin{equation}
\label{eq:parts-L}
\int_{r}^{s}L^{a,b}f(x)g(x)\, dx=U(f,g)(x)\Big|_{x=r}^{x=s}+\int_{r}^{s}f(x)L^{a,b}g(x)\, dx,
\end{equation}
holds with
\[
U(f,g)(x)=f(x)\frac{dg}{dx}(x)-g(x)\frac{df}{dx}(x).
\]
Note that we can connect $L^{a,b}$ and $L^{c,d}$ by mean of
\begin{equation}
\label{eq:rel-L}
L^{a,b}=L^{c,d}+W_{a,b}^{c,d},
\end{equation}
where
\[
W_{a,b}^{c,d}(x)=\frac{a^2-c^2}{4\sin^2(x/2)}+\frac{b^2-d^2}{4\cos^2(x/2)}.
\]

One of the main tools to estimate the kernel $K_{\alpha,\beta}^{\gamma,\delta}(n,m)$ is the bound (see \cite[eq.~(2.8)]{Muckenhoupt})
\begin{equation}
\label{eq:unif-bound}
|p_n^{(a,b)}(x)|\le C \begin{cases}
(n+1)^{a+1/2}(\sin x/2)^{a+1/2}, & 0<x<1/(n+1),\\
1, & 1/(n+1)\le x\le \pi-1/(n+1),\\
(n+1)^{b+1/2}(\cos x/2)^{b+1/2}, & \pi-1/(n+1)<x<\pi,
\end{cases}
\end{equation}
which holds for $n\in\mathbb{N}$ and $a,b>-1$.

\begin{proof}[Proof of Proposition \ref{propo:size}]
First, note that we may suppose that $n>m$ due to the symmetry.

We decompose the kernel $K_{\alpha,\beta}^{\gamma,\delta}(n,m)$ according to the intervals $I_1=(0,1/(n+1))$, $I_2=[1/(n+1), \pi-1/(n+1)]$, $I_3=(\pi-1/(n+1), \pi)$ and we denote the corresponding integrals by $K_1(n,m)$, $K_2(n,m)$, and $K_3(n,m)$, respectively. For $K_1$ and $K_3$, by \eqref{eq:unif-bound}, we have that
\[
|K_1(n,m)|\le C (n+1)^{\gamma+1/2}(m+1)^{\alpha+1/2}\int_{I_1} (\sin x/2)^{\gamma+\alpha+1}\, dx\le \frac{C}{n},
\]
and
\[
|K_3(n,m)|\le C (n+1)^{\delta+1/2}(m+1)^{\beta+1/2}\int_{I_3} (\cos x/2)^{\delta+\beta+1}\, dx\le \frac{C}{n}.
\]

In order to estimate $K_{2}(n,m)$ we consider some cases. If $\lambda_n^{(\gamma,\delta)}=\lambda_m^{(\alpha,\beta)}$, we get $|n-m|\sim C$ and therefore we can use that
\begin{equation}
\label{eq_K2orto}
|K_2(n,m)|\leq \|p_n^{(\gamma,\delta)}\|_{L^2(0,\pi)}\|p_m^{(\alpha,\beta)}\|_{L^2(0,\pi)}=1
\end{equation}
to obtain the result.

If $\lambda_n^{(\gamma,\delta)}\not=\lambda_m^{(\alpha,\beta)}$, $\alpha+\beta\not=\gamma+\delta$, and $n-m\leq |\alpha+\beta-\gamma-\delta|$, \eqref{eq_K2orto} also leads to the result.

Otherwise, we use \eqref{eq:eigen-eq}, \eqref{eq:parts-L} (specified for $f(x)=p_n^{(\gamma,\delta)}(x)$, $g(x)=p_m^{(\alpha,\beta)}(x)$, $a=\gamma$, $b=\delta$, $r=1/(n+1)$, and $s=\pi-1/(n+1)$), and \eqref{eq:rel-L} (with $a=\gamma$, $b=\delta$, $c=\alpha$, and $d=\beta$), and denote
\[
J(n,m)=\int_{I_2}W_{\gamma,\delta}^{\alpha,\beta}(x)p_n^{(\gamma,\delta)}(x)p_m^{(\alpha,\beta)}(x)\, dx,
\]
and
\[
S(n,m)=U(p_n^{(\gamma,\delta)},p_m^{(\alpha,\beta)})(x)\Big|_{x=1/(n+1)}^{x=\pi-1/(n+1)}.
\]
We obtain that
\begin{align*}
\lambda_n^{(\gamma,\delta)}K_2(n,m)&=\int_{I_2}L^{\gamma,\delta}p_n^{(\gamma,\delta)}(x)p_m^{(\alpha,\beta)}(x)\, dx
\\&=S(n,m)+\int_{I_2}p_n^{(\gamma,\delta)}(x)L^{\gamma,\delta}p_m^{(\alpha,\beta)}(x)\, dx\\&=S(n,m)+\lambda_m^{(\alpha,\beta)}K_{2}(n,m)+J(n,m)
\end{align*}
and hence
\begin{equation}
\label{eq:k-eigen}
K_2(n,m)=\frac{S(n,m)+J(n,m)}{\lambda_n^{(\gamma,\delta)}-\lambda_{m}^{(\alpha,\beta)}}.
\end{equation}
Taking the operator
\[
\Psi=\frac{d}{dx}-\frac{2a+1}{4}\cot\frac{x}{2}+\frac{2b+1}{4}\tan\frac{x}{2},
\]
it is known that (see \cite[7.7]{Now-Stem})
\begin{equation}
\label{eq:delta-pn}
\Psi p_n^{(a,b)}=-\sqrt{n(n+a+b+1)}p_{n-1}^{(a+1,b+1)},
\end{equation}
for $n\neq0$, and $\Psi p_{0}^{(a,b)}=0$, $a,b>-1$. Then, if $n\neq0$,
\begin{multline}
\label{eq:der-pn}
\frac{d p_n^{(a,b)}}{dx}(x)=-\sqrt{n(n+a+b+1)}p_{n-1}^{(a+1,b+1)}(x)\\+\left(\frac{2a+1}{4}\cot\frac{x}{2}-\frac{2b+1}{4}\tan\frac{x}{2}\right)p_{n}^{(a,b)}(x)
\end{multline}
and
\begin{equation}
\label{eq:der-pn0}
\frac{d p_0^{(a,b)}}{dx}(x)=\left(\frac{2a+1}{4}\cot\frac{x}{2}-\frac{2b+1}{4}\tan\frac{x}{2}\right)p_{0}^{(a,b)}(x).
\end{equation}

The use of these identities and the bound \eqref{eq:unif-bound} show that
\begin{equation}
\label{eq:bound-S}
|S(n,m)|\le C (n+m).
\end{equation}

Finally, we prove the estimate
\begin{equation}
\label{eq:bound-J}
|J(n,m)|\le C (n+m).
\end{equation}
We decompose $J(n,m)$ into four integrals, denoted by  $J_1(n,m)$, $J_2(n,m)$, $J_3(n,m)$, $J_4(n,m)$, over the intervals $U_1=[1/(n+1),1/(m+1))$, $U_2=[1/(m+1),\pi/2]$, $U_3=(\pi/2,\pi-1/(m+1)]$, and $U_4=(\pi-1/(m+1),\pi-1/(n+1)]$. We will only work with $J_1$ and $J_2$ because the other ones can be treated in a similar way. By \eqref{eq:unif-bound} and since $\alpha\ge -1/2$, we get
\[
|J_1(n,m)|\le C (m+1)^{\alpha+1/2}\int_{U_1} (\sin x/2)^{\alpha-3/2}\, dx\le C\int_{1/(n+1)}^{\pi/2}\frac{dx}{x^2}\le C n
\]
and
\[
|J_2(n,m)|\le C\int_{U_{2}}\frac{dx}{x^2}\le C (m+1).
\]
If $m=0$, $|J_2(n,m)|\le C$ and if $m>0$, $|J_2(n,m)|\le Cm$; therefore the bound \eqref{eq:bound-J} holds.

So, we conclude that
\[
|K_2(n,m)|\le \frac{C}{|n-m|}
\]
from \eqref{eq:k-eigen}, \eqref{eq:bound-S}, and \eqref{eq:bound-J}, and the proof is completed.
\end{proof}

\section{Proof of the Proposition \ref{propo:reg}}
In the proof of Proposition \ref{propo:reg} we will use the two following technical lemmas. Their proofs are postponed to next section.

\begin{lem}
\label{lem:bound-diff}
  Let $n\in\mathbb{N}$ and $a,b> -1$, then
\begin{multline}
\label{eq:bound-diff}
|p_{n+2}^{(a,b)}(x)-p_{n}^{(a,b)}(x)|\\\le C
\begin{cases}
(n+1)^{a-1/2}(\sin x/2)^{a+1/2}, & 0<x<1/(n+1),\\
\sin x/2\cos x/2, & 1/(n+1)\le x \le \pi-1/(n+1),\\
(n+1)^{b-1/2}(\cos x/2)^{b+1/2}, & \pi-1/(n+1)<x<\pi.
\end{cases}
\end{multline}
\end{lem}

\begin{lem}
\label{lem:bound-diff-der}
  Let $n\in\mathbb{N}$ and $a,b> -1$, then
\begin{multline}
\label{eq:bound-diff-der}
|(p_{n+2}^{(a,b)}-p_{n}^{(a,b)})'(x)|\\\le C
\begin{cases}
(n+1)^{a-1/2}(\sin x/2)^{a-1/2}, & 0<x<1/(n+1),\\
(n+1)\sin x/2\cos x/2, & 1/(n+1)\le x \le  \pi-1/(n+1),\\
(n+1)^{b-1/2}(\cos x/2)^{b-1/2}, &\pi-1/(n+1)<x<\pi.
\end{cases}
\end{multline}
\end{lem}

\begin{proof}[Proof of Proposition \ref{propo:reg}]
We focus on the proof of \eqref{eq:reg-cond-1} because \eqref{eq:reg-cond-2} can be deduced in a similar way.

We will consider two cases $n>m$ and $m>n$. For the first one, we decompose the difference $K_{\alpha,\beta}^{\gamma,\delta}(n+2,m)-K_{\alpha,\beta}^{\gamma,\delta}(n,m)$ according to the intervals $I_1=(0,1/(n+1))$, $I_2=[1/(n+1),\pi-1/(n+1)]$, and $I_3=(\pi-1/(n+1),\pi)$, and we denote the corresponding integrals by $\mathcal{K}_1(n,m)$, $\mathcal{K}_2(n,m)$, and $\mathcal{K}_3(n,m)$.  From \eqref{eq:unif-bound}, \eqref{eq:bound-diff}, and the facts that $\alpha,\beta\geq-1/2$, we have that
\[
|\mathcal{K}_1(n,m)|\le C (n+1)^{\gamma-1/2}(m+1)^{\alpha+1/2}\int_{I_1}(\sin x/2)^{\gamma+\alpha+1}\, dx\le \frac{C}{n^2}
\]
and
\[
|\mathcal{K}_3(n,m)|\le C (n+1)^{\delta-1/2}(m+1)^{\beta+1/2}\int_{I_3}(\cos x/2)^{\delta+\beta+1}\, dx\le \frac{C}{n^2}.
\]
and this is enough for our purpose.

Now, in order to estimate $\mathcal{K}_2(n,m)$ we consider some situations. The cases a) $\lambda_{n}^{(\gamma,\delta)}=\lambda_{m}^{(\alpha,\beta)}$, b) $\lambda_{n}^{(\gamma,\delta)}\neq\lambda_{m}^{(\alpha,\beta)}$, $\lambda_{n+2}^{(\gamma,\delta)}=\lambda_{m}^{(\alpha,\beta)}$, and c) $\lambda_{n}^{(\gamma,\delta)}\neq\lambda_{m}^{(\alpha,\beta)}$, $\lambda_{n+2}^{(\gamma,\delta)}\neq\lambda_{m}^{(\alpha,\beta)}$, $\alpha+\beta\neq\gamma+\delta$, $n-m\leq|\alpha+\beta-\gamma-\delta|$ are easy to handle with.

Otherwise, to avoid cumbersome notations we denote
\[
\mathcal{J}(n,m)=\int_{I_2}W_{\gamma,\delta}^{\alpha,\beta}(x)(p_{n+2}^{(\gamma,\delta)}(x)-p_n^{(\gamma,\delta)}(x))p_m^{(\alpha,\beta)}(x)\, dx
\]
and
\[
\mathcal{S}(n,m)=U(p_{n+2}^{(\gamma,\delta)}-p_n^{(\gamma,\delta)},p_m^{(\alpha,\beta)})(x)\Big|_{x=1/(n+1)}^{x=\pi-1/(n+1)}.
\]
Then, using \eqref{eq:k-eigen} and following the notation given in the proof of Proposition~\ref{propo:size}, we have
\begin{align}
\label{eq:kernel-diff}
\mathcal{K}_2(n,m)&=\frac{S(n+2,m)+J(n+2,m)}{\lambda_{n+2}^{(\gamma,\delta)}-\lambda_{m}^{(\alpha,\beta)}}
-\frac{S(n,m)+J(n,m)}{\lambda_n^{(\gamma,\delta)}-\lambda_{m}^{(\alpha,\beta)}}\\\notag
&=\frac{\mathcal{S}(n,m)+\mathcal{J}(n,m)}{\lambda_{n+2}^{(\gamma,\delta)}-\lambda_{m}^{(\alpha,\beta)}}
-\frac{(\lambda_{n+2}^{(\gamma,\delta)}-\lambda_{n}^{(\gamma,\delta)})(S(n,m)+J(n,m))}
{(\lambda_{n+2}^{(\gamma,\delta)}-\lambda_{m}^{(\alpha,\beta)})(\lambda_n^{(\gamma,\delta)}-\lambda_{m}^{(\alpha,\beta)})}
\\\notag
&=\frac{\mathcal{S}(n,m)+\mathcal{J}(n,m)}{\lambda_{n+2}^{(\gamma,\delta)}-\lambda_{m}^{(\alpha,\beta)}}
-\frac{2(2n+\gamma+\delta+3)(S(n,m)+J(n,m))}
{(\lambda_{n+2}^{(\gamma,\delta)}-\lambda_{m}^{(\alpha,\beta)})(\lambda_n^{(\gamma,\delta)}-\lambda_{m}^{(\alpha,\beta)})}.
\end{align}

The terms $S(n,m)$ and $J(n,m)$ can be treated proceeding as in Proposition \ref{propo:size}, so we obtain
\begin{equation}
\label{eq:kernel-diff-1}
\left|\frac{2(2n+\gamma+\delta+3)(S(n,m)+J(n,m))}
{(\lambda_{n+2}^{(\gamma,\delta)}-\lambda_{m}^{(\alpha,\beta)})(\lambda_n^{(\gamma,\delta)}-\lambda_{m}^{(\alpha,\beta)})}\right|\le \frac{C}{|n-m|^2}.
\end{equation}

On the other hand, from \eqref{eq:unif-bound}, \eqref{eq:der-pn} (for $m\neq0$ or \eqref{eq:der-pn0} if $m=0$), \eqref{eq:bound-diff}, and \eqref{eq:bound-diff-der}, we deduce that
\[
|\mathcal{S}(n,m)|\le C.
\]
This fact implies the estimate
\begin{equation}
\label{eq:kernel-diff-11}
\left|\frac{\mathcal{S}(n,m)}{\lambda_{n+2}^{(\gamma,\delta)}-\lambda_{m}^{(\alpha,\beta)}}\right|\le \frac{C}{|n-m|^2}.
\end{equation}

Finally, in order to analyse $\mathcal{J}(n,m)$, we denote
\[
\overline{\mathcal{S}}(n,m)=U(p_m^{(\alpha,\beta)},W_{\gamma,\delta}^{\alpha,\beta}(p_{n+2}^{(\gamma,\delta)}-p_n^{(\gamma,\delta)}))(x)\Big|_{x=1/(n+1)}^{x=\pi-1/(n+1)},
\]
\[
T_1(n,m)=\int_{I_2} ((W_{\gamma,\delta}^{\alpha,\beta})^2(x)+(W_{\gamma,\delta}^{\alpha,\beta})''(x))(p_{n+2}^{(\gamma,\delta)}(x)-p_{n}^{(\gamma,\delta)}(x))p_m^{(\alpha,\beta)}(x)\, dx,
\]
and
\[
T_2(n,m)=\int_{I_2} (W_{\gamma,\delta}^{\alpha,\beta})'(x)(p_{n+2}^{(\gamma,\delta)}-p_{n}^{(\gamma,\delta)})'(x)p_m^{(\alpha,\beta)}(x)\, dx.
\]
Therefore, using \eqref{eq:parts-L} and the identity
\[
L^{\alpha,\beta}(W_{\gamma,\delta}^{\alpha,\beta}f)=W_{\gamma,\delta}^{\alpha,\beta}L^{\gamma,\delta}f-((W_{\gamma,\delta}^{\alpha,\beta})^2
+(W_{\gamma,\delta}^{\alpha,\beta})'')f-2(W_{\gamma,\delta}^{\alpha,\beta})'f',
\]
we can deduce that
\begin{align*}
\lambda_m^{(\alpha,\beta)}\mathcal{J}(n,m)&=\int_{I_2} W_{\gamma,\delta}^{\alpha,\beta}(x)(p_{n+2}^{(\gamma,\delta)}(x)-p_{n}^{(\gamma,\delta)}(x))L^{\alpha,\beta}p_m^{(\alpha,\beta)}(x)\, dx\\
&=
\overline{\mathcal{S}}(n,m)+\int_{I_2} L^{\alpha,\beta}(W_{\gamma,\delta}^{\alpha,\beta}(p_{n+2}^{(\gamma,\delta)}-p_{n}^{(\gamma,\delta)}))(x)p_m^{(\alpha,\beta)}(x)\, dx\\
&=
\overline{\mathcal{S}}(n,m) +\int_{I_2} W_{\gamma,\delta}^{\alpha,\beta}(x)L^{\gamma,\delta}(p_{n+2}^{(\gamma,\delta)}-p_{n}^{(\gamma,\delta)})(x)p_m^{(\alpha,\beta)}(x)\, dx\\
&\quad-T_1(n,m)-2T_2(n,m)
\end{align*}
Now note that
\[
L^{\gamma,\delta}(p_{n+2}^{(\gamma,\delta)}-p_{n}^{(\gamma,\delta)})(x)=\lambda_{n+2}^{(\gamma,\delta)}(p_{n+2}^{(\gamma,\delta)}(x)-p_{n}^{(\gamma,\delta)}(x))
+(\lambda_{n+2}^{(\gamma,\delta)}-\lambda_{n}^{(\gamma,\delta)})p_{n}^{(\gamma,\delta)}(x)
\]
and hence
\begin{multline*}
\lambda_m^{(\alpha,\beta)}\mathcal{J}(n,m)=\overline{\mathcal{S}}(n,m)+\lambda_{n+2}^{(\gamma,\delta)}\mathcal{J}(n,m)\\
+(\lambda_{n+2}^{(\gamma,\delta)}-\lambda_n^{(\gamma,\delta)})J(n,m)-T_1(n,m)-2T_2(n,m).
\end{multline*}
Thus,
\begin{equation}\label{eq:kernel-diff-2}
\frac{\mathcal{J}(n,m)}{\lambda_{n+2}^{(\gamma,\delta)}-\lambda_{m}^{(\alpha,\beta)}}=
\frac{-\overline{\mathcal{S}}(n,m)-2(2n+\gamma+\delta+3)J(n,m)+T_1(n,m)+2T_2(n,m)}{(\lambda_{n+2}^{(\gamma,\delta)}-\lambda_m^{(\alpha,\beta)})^2}.
\end{equation}

It is again straightforward from Proposition \ref{propo:size} to deduce that
\[
\left|\frac{2(2n+\gamma+\delta+3)}{(\lambda_m^{(\alpha,\beta)}-\lambda_{n+2}^{(\gamma,\delta)})^2}J(n,m)\right|\le \frac{C}{|n-m|^2}.
\]
In this way, by \eqref{eq:kernel-diff}, \eqref{eq:kernel-diff-1}, \eqref{eq:kernel-diff-11}, and \eqref{eq:kernel-diff-2}, showing that
\[
|\overline{\mathcal{S}}(n,m)|+|T_1(n,m)|+|T_2(n,m)|\le Cn^2
\]
the proof of \eqref{eq:reg-cond-1}  for $n>m$ will be completed.

We use \eqref{eq:unif-bound}, \eqref{eq:der-pn} (for $m\neq0$ or \eqref{eq:der-pn0} if $m=0$), \eqref{eq:bound-diff}, and \eqref{eq:bound-diff-der}, and then,
\[
|\overline{\mathcal{S}}(n,m)|\le C n^2.
\]
In order to estimate $T_1(n,m)$ and $T_2(n,m)$ we will focus on the interval $[1/(n+1),\pi/2]$ because the complementary interval $(\pi/2,\pi-1/(n+1)]$ can be studied in a similar way. The corresponding integrals over $(1/(n+1),\pi/2]$ will be denoted by $\mathcal{T}_1(n,m)$ and $\mathcal{T}_2(n,m)$. First, using that
\[
(W_{\gamma,\delta}^{\alpha,\beta})^2(x)+|(W_{\gamma,\delta}^{\alpha,\beta})''(x)|\le C (\sin x/2)^{-4}, \qquad x\in [1/(n+1),\pi/2],
\]
we have
\[
T_1(n,m)|\le C\int_{1/(n+1)}^{\pi/2}|p_{n+2}^{(\gamma,\delta)}(x)-p_n^{(\gamma,\delta)}(x)||p_m^{(\alpha,\beta)}(x)|\frac{dx}{(\sin x/2)^4}.
\]
Then, using \eqref{eq:unif-bound}, \eqref{eq:bound-diff}, and the restriction $\alpha\ge -1/2$, we obtain that
\begin{align*}
|T_1(n,m)|
&\le C \left((m+1)^{\alpha+1/2} \int_{1/(n+1)}^{1/(m+1)} (\sin x/2)^{\alpha-5/2}\, dx+ \int_{1/(m+1)}^{\pi/2}\frac{dx}{(\sin x/2)^3}\right)\\&\le
C\int_{1/(n+1)}^{\pi/2}\frac{dx}{x^3}\le C n^2.
\end{align*}
Finally, using the estimate
\[
|(W_{\gamma,\delta}^{\alpha,\beta})'(x)|\le C (\sin x/2)^{-3}, \qquad x\in [1/(n+1),\pi/2],
\]
\eqref{eq:unif-bound}, \eqref{eq:bound-diff-der}, and $\alpha\ge -1/2$,
we deduce that
\begin{align*}
|T_2(n,m)|
&\le C (n+1) \left((m+1)^{\alpha+1/2} \int_{1/(n+1)}^{1/(m+1)} (\sin x/2)^{\alpha-3/2}\, dx\right.\\&\kern70pt\left.+ \int_{1/(m+1)}^{\pi/2}\frac{dx}{(\sin x/2)^2}\right)\\&\le
C (n+1)\int_{1/(n+1)}^{\pi/2}\frac{dx}{x^2}\le C n^2.
\end{align*}

We deal now with the case $m>n$. At this point, we find convenient to reset the notation taken for the case $n>m$.

Again, we decompose the difference $K_{\alpha,\beta}^{\gamma,\delta}(n+2,m)-K_{\alpha,\beta}^{\gamma,\delta}(n,m)$ into three integrals, $\mathcal{K}_{1}(n,m)$, $\mathcal{K}_{2}(n,m)$, and $\mathcal{K}_{3}(n,m)$ over the intervals $I_1=(0,1/(n+1))$, $I_2=[1/(n+1),\pi-1/(n+1)]$, and $I_3=(\pi-1/(n+1),\pi)$, respectively.

Using \eqref{eq:unif-bound}, \eqref{eq:bound-diff}, and the fact that $m/2\leq n\leq 2m$, we have that
\begin{multline*}
|\mathcal{K}_{1}(n,m)| \leq C (n+1)^{\gamma-1/2}\Big((m+1)^{\alpha+1/2}\int_{0}^{1/(m+1)} (\sin x/2)^{\gamma+\alpha+1}\,dx \\+ \int_{1/(m+1)}^{1/(n+1)} (\sin x/2)^{\gamma+1/2}\,dx\Big) \leq \frac{C}{m^{2}},
\end{multline*}
and
\begin{multline*}
|\mathcal{K}_{3}(n,m)| \leq C (n+1)^{\delta-1/2}\Big(\int_{\pi-1/(n+1)}^{\pi-1/(m+1)} (\cos x/2)^{\delta+1/2}\,dx \\+ (m+1)^{\beta+1/2}\int_{\pi-1/(m+1)}^{\pi} (\cos x/2)^{\delta+\beta+1}\,dx\Big) \leq \frac{C}{m^{2}}.
\end{multline*}

We analyse now $\mathcal{K}_{2}(n,m)$. Again, the cases a) $\lambda_{n}^{(\gamma,\delta)}=\lambda_{m}^{(\alpha,\beta)}$, b) $\lambda_{n}^{(\gamma,\delta)}\neq\lambda_{m}^{(\alpha,\beta)}$, $\lambda_{n+2}^{(\lambda,\delta)}=\lambda_{m}^{(\alpha,\beta)}$, c) $\lambda_{n}^{(\gamma,\delta)}\neq\lambda_{m}^{(\alpha,\beta)}$, $\lambda_{n+2}^{(\lambda,\delta)}\neq\lambda_{m}^{(\alpha,\beta)}$, $\alpha+\beta\neq\gamma+\delta+4$, $m-n\leq|\alpha+\beta-\gamma-\delta-4|$ can be treated in an elementary way.

Otherwise, note first that \eqref{eq:kernel-diff} still holds for $m>n$.
Following the same procedure as in the proof of Proposition~\ref{propo:size} we have that
\begin{equation*}
|S(n,m)| \leq C(n+m).
\end{equation*}
The bound \eqref{eq:unif-bound} leads to
\begin{equation*}
|J(n,m)| \leq C\int_{1/(n+1)}^{\pi/2} (\sin x/2)^{-2}\,dx \leq C(n+1).
\end{equation*}
Thus, if $n=0$, $|J(n,m)|\leq C$ and if $n>0$, $|J(n,m)|\leq Cn$.

Furthermore, the use of \eqref{eq:unif-bound}, \eqref{eq:der-pn}, \eqref{eq:bound-diff}, \eqref{eq:bound-diff-der}, and the fact that $m/2\leq n\leq 2m$ imply that
\begin{equation*}
|\mathcal{S}(n,m)| \leq C.
\end{equation*}

Now we analyse $\mathcal{J}(n,m)$. We note that \eqref{eq:kernel-diff-2} still holds in this case. We have already studied the term $J(n,m)$ so we will focus on the remaining ones. By \eqref{eq:unif-bound}, \eqref{eq:der-pn} (for $m\neq0$ or \eqref{eq:der-pn0} if $m=0$), \eqref{eq:bound-diff}, and \eqref{eq:bound-diff-der}, we have that
\begin{equation*}
|\overline{\mathcal{S}}(n,m)| \leq C m^{2}.
\end{equation*}
Finally, treating $T_{1}(n,m)$ and $T_{2}(n,m)$ as in the case $n>m$ we obtain that $|T_{1}(n,m)|\leq C m^{2}$ and $|T_{2}(n,m)|\leq C m^{2}$.

From all of these bounds, \eqref{eq:reg-cond-1} follows for $m>n$ and then the proof is completed.
\end{proof}

\section{Proof of Lemmas \ref{lem:bound-diff} and \ref{lem:bound-diff-der}}
\begin{proof}[Proof of Lemma \ref{lem:bound-diff}] It is enough the result for $0<x\le \pi/2$ because for $\pi/2<x<\pi$ we can use the identity $p_n^{(a,b)}(\pi-z)=(-1)^n p_n^{(b,a)}(z)$, for $0<z<\pi$ to obtain the required inequality. That is, we will prove that for $0<x\leq\pi/2$,
\begin{multline}
\label{eq:bound-diff-2}
|p_{n+2}^{(a,b)}(x)-p_{n}^{(a,b)}(x)|\\\le C
\begin{cases}
(n+1)^{a-1/2}(\sin x/2)^{a+1/2}(\cos x/2)^{b+1/2}, & 0<x<1/(n+1),\\
\sin x/2, & 1/(n+1)\le x \le \pi/2.
\end{cases}
\end{multline}

From the definition of $p_{n}^{(a,b)}$ we have
\begin{multline}
\label{eq:lem-1}
p_{n+2}^{(a,b)}(x)-p_n^{(a,b)}(x)=\left(\frac{w_{n+2}^{(a,b)}}{w_n^{(a,b)}}-1\right)p_n^{(a,b)}(x)\\+(\sin x/2)^{a+1/2}(\cos x/2)^{b+1/2}w_{n+2}^{(a,b)}(P_{n+2}^{(a,b)}(\cos x)-P_{n}^{(a,b)}(\cos x)).
\end{multline}
By using that
\begin{equation}
\label{eq:asym-ratio}
\left|\frac{w_{n+2}^{(a,b)}}{w_n^{(a,b)}}-1\right|\leq \frac{C}{n+1}
\end{equation}
and the estimate \eqref{eq:unif-bound} (note that $1/(n+1)\le C \sin x/2$ if $1/(n+1)\leq x\leq\pi/2$), we conclude that
\begin{multline}
\label{eq:lem-2}
\left|\left(\frac{w_{n+2}^{(a,b)}}{w_n^{(a,b)}}-1\right)p_n^{(a,b)}(x)\right|
\\\le C \begin{cases}
(n+1)^{a-1/2}(\sin x/2)^{a+1/2}(\cos x/2)^{b+1/2}, & 0<x<1/(n+1),\\
\sin x/2\cos x/2, & 1/(n+1)\le x \le \pi/2.
\end{cases}
\end{multline}

For the remaining part on the right hand side of \eqref{eq:lem-1} we use the following identity (obtained from \cite[18.9.6]{NIST})
\[
-\frac{2n+a+b+2}{2}(1-z)P_n^{(a+1,b)}(z)+aP_n^{(a,b)}(z)=(n+1)(P_{n+1}^{(a,b)}(z)-P_{n}^{(a,b)}(z))
\]
for $z\in (-1,1)$. We obtain by means of it that
\begin{multline}
(\sin x/2)^{a+1/2}(\cos x/2)^{b+1/2}w_{n+2}^{(a,b)}|P_{n+1}^{(a,b)}(\cos x)-P_{n}^{(a,b)}(\cos x)|\\\le \frac{(2n+a+b+2)}{n+1}\frac{w_{n+2}^{(a,b)}}{w_{n}^{(a+1,b)}}(\sin x/2)|p_n^{(a+1,b)}(x)|+\frac{|a|}{n+1}\frac{w_{n+2}^{(a,b)}}{w_{n}^{(a,b)}}|p_n^{(a,b)}(x)|.
\end{multline}
In this way, \eqref{eq:unif-bound} implies that the latter is bounded by
\begin{equation}
\label{eq:lem-3}
C
\begin{cases}
(n+1)^{a-1/2}(\sin x/2)^{a+1/2}, & 0<x<1/(n+1),\\
\sin x/2, & 1/(n+1)\le x \le \pi/2,
\end{cases}
\end{equation}
The same bound holds for the term
\begin{equation*}
(\sin x/2)^{a+1/2}(\cos x/2)^{b+1/2}w_{n+2}^{(a,b)}|P_{n+2}^{(a,b)}(\cos x)-P_{n+1}^{(a,b)}(\cos x)|.
\end{equation*}

Then, \eqref{eq:bound-diff-2} follows from \eqref{eq:lem-1}, \eqref{eq:lem-2}, and \eqref{eq:lem-3}.
\end{proof}

\begin{proof}[Proof of Lemma \ref{lem:bound-diff-der}]
First, we suppose that $0<x\leq\pi/2$ and we prove that
\begin{multline}
\label{eq:bound-diff-der-2}
|(p_{n+2}^{(a,b)}-p_{n}^{(a,b)})'(x)|\\\le C
\begin{cases}
(n+1)^{a-1/2}(\sin x/2)^{a-1/2}, & 0<x<1/(n+1),\\
(n+1)\sin x/2\cos x/2, & 1/(n+1)\le x \le  \pi/2.
\end{cases}
\end{multline}

If $n\neq0$, by using \eqref{eq:der-pn}, we get
\[
(p_{n+2}^{(a,b)}-p_n^{(a,b)})'(x)=S_1(n,a,b,x)-S_2(n,a,b,x)-S_3(n,a,b,x),
\]
where
\[
S_1(n,a,b,x)=\left(\frac{2a+1}{4}\cot\frac{x}{2}-\frac{2b+1}{4}\tan\frac{x}{2}\right)(p_{n+2}^{(a,b)}(x)-p_n^{(a,b)}(x)),
\]
\[
S_2(n,a,b,x)=\sqrt{(n+2)(n+a+b+3)}(p_{n+1}^{(a+1,b+1)}(x)-p_{n-1}^{(a+1,b+1)}(x)),
\]
and
\[
S_3(n,a,b,x)=(\sqrt{(n+2)(n+a+b+3)}-\sqrt{n(n+a+b+1)})p_{n-1}^{(a+1,b+1)}(x).
\]

The following bounds are a consequence of \eqref{eq:unif-bound} and \eqref{eq:bound-diff}.
\begin{equation}
|S_1(n,a,b,x)|\\\le C \begin{cases}
(n+1)^{a-1/2}(\sin x/2)^{a-1/2}, & 0<x<1/(n+1),\\
\cos x/2, & 1/(n+1)\le x \le \pi/2,
\end{cases}
\end{equation}
\[
|S_2(n,a,b,x)|\le C \begin{cases}
(n+1)^{a+3/2}(\sin x/2)^{a+3/2}, & 0<x<1/(n+1),\\
n\sin x/2\cos x/2, & 1/(n+1)\le x \le \pi/2,
\end{cases}
\]
and
\[
|S_3(n,a,b,x)|\le C \begin{cases}
n^{a+3/2}(\sin x/2)^{a+3/2}, & 0<x<1/(n+1),\\
1, & 1/(n+1)\le x \le \pi/2,
\end{cases}
\]
and the result follows. If $n=0$, we use \eqref{eq:der-pn0} instead of \eqref{eq:der-pn} and the result is also valid.

In the case $\pi/2<x<\pi$ we use again the identity $p_n^{(a,b)}(\pi-z)=(-1)^n p_n^{(b,a)}(z)$, with $0<z<\pi$, to deduce the bound and the proof is completed.
\end{proof}



\end{document}